\documentclass{birkjour}
\usepackage{amsmath, amssymb, amscd, amsthm, amsfonts}
\usepackage{hyperref, relsize}

\theoremstyle{plain}
\newtheorem{thm}{Theorem}[section]
\newtheorem{cor}[thm]{Corollary}
\newtheorem{lem}[thm]{Lemma}

\theoremstyle{definition}
\newtheorem{defn}[thm]{Definition}
\theoremstyle{remark}
\newtheorem{rem}[thm]{Remark}
\newtheorem*{ex}{Example}
\numberwithin{equation}{section}

\begin{document}

%
%
%
%
%
%
%
%
%

\def \al {{\alpha}}
\def \T {{\mathbb T}}
\def \D {{\mathbb D}}
\def \C {{\mathbb C}}
\def \R {{\mathbb R}}
\def \N {{\mathbb N}}
\def \Z {{\mathbb Z}}

\title[Invariant subspaces and sub-Hilbert spaces]{Helson-Lowdenslager and de Branges type theorems in the setting of continuous rotationally symmetric norms}

\author[A. Singh]{Apoorva Singh}

\address{Department of Mathematics\\
School of Natural Sciences\\
Shiv Nadar Institution of Eminence (Deemed to be University)\\
Gautam Buddha Nagar - 201314.\\
Uttar Pradesh, India}

\email{as721@snu.edu.in}

\author[N. Sahni]{Niteesh Sahni}
\address{Department of Mathematics\br
School of Natural Sciences\br
Shiv Nadar Institution of Eminence (Deemed to be University)\br
Gautam Buddha Nagar - 201314.\br
Uttar Pradesh, India}
\email{niteesh.sahni@snu.edu.in}
\subjclass{Primary 47A15; Secondary 30H10, 47B38 }

\keywords{ Rotationally symmetric norm, Simply invariant subspace, Lebesgue space, Hardy space, Sub-Hilbert spaces.}

\begin{abstract}
A Helson-Lowdenslager type result has been proved by Chen \cite{chen2017general} in the context of Lebesgue spaces of the unit circle equipped with a continuous rotationally symmetric norm by studying the simply invariant subspaces of the operator of multiplication by the coordinate function $z$. In this paper, we generalize Chen's result by obtaining a description of simply invariant subspaces for multiplication by $z^n$. A de Branges type result is also proved for Hardy spaces equipped with continuous rotationally symmetric norms. 

\end{abstract}

\maketitle
\section{Introduction}\label{sec1}

Two problems in the theory of analytic functions on the unit disk that have been a center of rigorous research are Beurling's theorem \cite{Beurling1948} which characterizes subspaces of the Hardy space $H^2$ that are invariant under $T_z$-- the operator of multiplication by the coordination function $z$, and de Branges' theorem \cite{de2015square} which characterizes contractively contained sub-Hilbert spaces in $H^2$ that are also invariant under $T_z$. 
\par
Beurling's theorem was later generalized by Helson and Lowdenslager \cite{invariant_subspaces_Helson} in form of obtaining the simply invariant subspaces of $T_z$ on the Lebesgue space $L^2$ of the unit circle. The doubly invariant subspaces of $L^2$ under the operator $T_z$ were obtained by Weiner \cite{helson2017lectures}. The simplicity of the arguments allowed for the generalizations to all $L^p$ spaces ($0< p \leq \infty$) \cite{Srinivasan1963simply,Srinivasan1964doubly}. Ever since, the Helson--Lowdenslager theorem has been proved in many different settings. For example, on the torus, subspaces of $L^2$ invariant under doubly commuting isometries have been obtained in \cite{Mandrekar1988Polydisc}, and further extended to all $L^p$ spaces in \cite{Redett2005Polydisc}. The vector-valued extensions of the Beurling and the Helson--Lowdenslager theorems obtained by Lax \cite{lax1959translation} and Halmos \cite{Halmos1961Shifts} were significant breakthroughs which sparked off a significant body of research in this direction. Motivated by the results of Lax and Halmos, invariance under the operator $T_{z^n}$ on $H^p$ and $L^p$ spaces have been studied in \cite{sahni2012lax, SnehLata2010finite}. In fact,  \cite{sahni2012lax} studies invariance under $T_B$-- the multiplication by a finite Blaschke factor $B(z)$ and also the common invariant subspaces of $T_{B^2}$ and $T_{B^3}$ on all $H^p$ spaces. 
\par
Coming to the study of sub-Hilbert spaces of Lebesgue spaces, it was established in \cite{Paulsen2001deBrangesL2} that a general de-Branges type result is not possible in $L^2$. However in \cite{Paulsen2001deBrangesL2}, a characterization is obtained under some conditions on the norm of the sub-Hilbert space and they went on to prove a characterization for $L^p$ ($p>2$). This result was extended to other $L^p$ spaces ($1 \leq p < 2$) in \cite{Redett2005SubLebesgue}. Redett's result was further generalized to all $L^p$ spaces in the context of the operator $T_{z^n}$ in \cite{SnehLata2010finite}.
\par
Recently in \cite{Chen2014Thesis}, Chen introduced a class of norms $\al$ over all measurable functions on the unit circle called the continuous rotationally symmetric norm, and defined the general Lebesgue and Hardy spaces associated with $\al$ denoted by $L^\alpha$ and $H^\alpha$, respectively. The space $L^\al$ is the $\al$-closure of $L^\infty$ and $H^\al$ is the $\al$-closure of $H^\infty$. The classical $p$-norms $\|.\|_p$ $(1 \leq p < \infty$) are particular examples of continuous rotationally symmetric norms. Interestingly, \cite{chen2017general} introduced a more general class of norms called the $\|.\|_1$-dominating normalized gauge norms on measurable functions on the unit circle and proved the Helson-Lowdenslager type result for the operator $T_z$ on Lebesgue spaces equipped with such norms. In fact, doubly invariant subspaces have also been studied in this general setting. For $H^\al$ under a $\|.\|_1$-dominating gauge norm, invariance under multiplication by a finite Blaschke factor has been studied in \cite{singh2022multiplication}. Further, sharp descriptions for invariant as well as common invariant subspaces under the rotationally symmetric norms have also been presented in \cite{singh2022multiplication}. 
\par
The aim of the present paper is two-fold. First, the Beurling type results of \cite{singh2022multiplication} have been extended to $L^\al$ - the Lebesgue space equipped with a continuous rotationally symmetric norm. In particular, we describe the closed subspaces of $L^\al$ which are simply invariant under the operator $T_{z^n}$. We point out that the invariance under $T_{z^n}$ on $L^p$ was first studied in \cite{SnehLata2010finite}, wherein, the invariant subspaces are first characterized for $L^2$ and then density arguments are given to lift the result to all $L^p$ spaces. The Wold decomposition plays a crucial role in the characterization for $L^2$ in addition to the orthogonal properties of $n$-unimodular functions. These fundamental properties of $n$-unimodular functions will be used in the proof for the general $L^\al$ case as well.  The arguments presented in \cite{chen2017general} for establishing a form for the $T_z$ simply invariant subspaces in $L^\al$ make use of the fact that an $L^\infty$ function $f$, satisfying $1/f \in L^\al$, can be factored into an $L^\infty$ function $g$ and an $H^\infty$ function $h$ such that $1/h\in H^\al$. Our proof is elementary in the sense that it does not rely on such factorization and the arguments work for the operator $T_{z^n}$. We define an $n$-unimodular matrix on tuple of $L^\infty$ functions and make use of the decomposition of  any $L^\al$ function into a direct sum of $L^\al(z^n)$ functions, defined as the $\al$-closure of $L^\infty(z^n)$, to prove our main theorem in section~\ref{sec3}.  This result generalizes the Helson--Lowdenslager's type theorems of  \cite{SnehLata2010finite}, \cite{chen2017general} and \cite{singh2022multiplication}. 
\par
Secondly, in section~\ref{sec4}, we examine the de Branges case in $H^\al$, that is, Hilbert spaces algebraically contained in $H^\al$ on which $T_z$ acts as an isometry. Our result generalizes the de Branges type result of \cite{Agrawal1995deBranges} in $H^p$ spaces.

\section{Notations and Preliminaries}\label{sec2}

Suppose $\D$ is the open unit disk and $\T$ is the unit circle on the complex plane $\C$. Let $m$ be a normalized Lebesgue measure on $\T$. The Lebesgue space $L^\infty$ consists of the essentially bounded complex valued measurable functions on $\T$, such that it is a Banach space under the essential supremum norm. We define a rotationally symmetric norm on $L^\infty$ as:
\begin{defn}(\cite{Chen2014Thesis})
Let $\al$ be a norm on $L^\infty$. $\al$ is a rotationally symmetric norm if:
\begin{enumerate}
    \item  $\al(1)=1$,
    \item $\al(|f|) = \al(f)$ for every $f \in L^\infty$, and
    \item $\al(f_w) = \al(f)$ for every $w \in \T$ and $f \in L^\infty$.
\end{enumerate}
Here $f_w : \T \to \C$ is defined as $f_w(z) = f(\overline{w}z)$.
\end{defn}
Moreover, the rotationally symmetric norm can also be extended over all complex valued measurable functions on $\T$ as
\begin{center}
$\alpha(f) $ = sup $\{ \alpha(s) : s $ is a simple function, $|s| \leq |f| \}$.
\end{center}
Furthermore, we say $\al$ is continuous if, for a sequence of measurable sets $\{E_n\}_{n=1}^\infty$,  
$m(E_n) \rightarrow 0^+$ we have $\al(\chi_{E_n}) \rightarrow 0.$ Here $\chi_{E_n}$ is the characteristic function on the set $E_n \subset \T$.
\par
One of the important examples of continuous rotationally symmetric norms, besides the $p$-norms ($1 \leq p < \infty$), is an Orlicz norm. We shall first define an Orlicz norm on $L^\infty$ and prove that it is a continuous rotationally symmetric norm. 
\par
Recall the definition of an Orlicz function. A non-decreasing convex function $\psi : [0, \infty] \to [0,\infty]$ is an Orlicz function such that $\psi(0)=0$ and $\psi(\infty)= \infty$. We additionally assume that $\psi$ is continuous at $0$, $\psi$ is a strictly increasing surjective function,
and $\lim_{x\to\infty} \frac{\psi(x)}{x} = \infty$. For a detailed study on Orcliz functions and their associated norms, we refer to \cite{Rao1991Theory}. 
\begin{ex}
One of the examples of a continuous rotationally symmetric norm is an Orlicz norm which is defined on $L^\infty$  as:
$$ \|f\|_\psi = \inf \left \{ \lambda > 0 : \frac{1}{\psi(1)}\int_{\T} \psi \left( \frac{\lvert f\rvert}{\lambda} \right) dm \leq 1 \right \}$$
where $\psi$ is an Orlicz function and $f \in L^\infty$.
\end{ex}
Below we justify that $\|.\|_\psi$ is a continuous rotationally symmetric norm on $L^\infty$.\\
We first show that the $\|.\|_\psi$ norm is well-defined on $L^\infty$. Whenever $f =0$, then $\psi(0)=0$ and $\|0\|_\psi=0$. Consider $0 \ne f \in L^\infty$, and observe that $\lvert f \rvert \leq \|f\|_\infty$ $a.e.$ We can write $\frac{\lvert f \rvert}{\|f\|_\infty} \leq 1$ $a.e.$, such that $\frac{1}{\psi(1)}\int_{\T} \psi \left( \frac{\lvert f\rvert}{\|f\|_\infty} \right) dm \leq 1$. Therefore, $\|f\|_\psi \leq \|f\|_\infty$.
\par
Now we prove that $\|.\|_\psi$ is a rotationally symmetric norm. Clearly, $\|1\|_\psi = \inf \left \{ \lambda > 0 : \frac{1}{\psi(1)}\int_{\T} \psi \left( \frac{ 1}{\lambda}  \right) dm \leq 1 \right \}$= $\inf \left \{ \lambda > 0 : \psi \left( \frac{ 1}{\lambda}  \right) \leq \psi(1) \right \}$. Since $\psi$ is strictly increasing on $[0,\infty]$, implies $\psi$ is a bijective function, therefore $\lambda \geq 1$ and $\|1\|_\psi = 1$. It is easy to show that $\|\lvert f \rvert \|_\psi = \|f\|_\psi$, and by changing the variable, it follows that $\|f_w\|_\psi = \|f\|_\psi$ for every $w \in \T$.
\par
Lastly, for the continuity of $\|.\|_\psi$,  suppose $\{E_n\}_{n=1}^\infty$ is a sequence of measurable sets in $\T$ such that $m(E_n) \to 0^+$ as $n \to \infty$.\\
We shall prove that $\|\chi_{E_n}\|_\psi \to 0$ as $n \to \infty$.\\
Since $m(E_n) \to 0^+$, it can be easily seen that $\chi_{E_n} \to 0$ $a.e.$ as $n \to \infty$. For any $\lambda > 0$, $\dfrac{ \chi_{E_n}}{\lambda} \to 0~a.e.$ It is well known that every real-valued convex function on an open interval is continuous. So, $\psi$ is continuous which gives $\psi \left ( \dfrac{ \chi_{E_n}  }{\lambda} \right ) \to 0~a.e.~\forall~\lambda > 0.$ So, for a given $\epsilon > 0$, there exists $K \in \N$ such that 
$$ \psi \left ( \dfrac{ \chi_{E_n}  }{\lambda} \right ) \leq \epsilon~a.e.~\forall~n \geq K.$$
This implies, 
\begin{align*}
 \dfrac{1}{\psi(1)} \int\limits_\T \psi \left ( \dfrac{\chi_{E_n} }{\lambda} \right ) dm &\leq \dfrac{1}{\psi(1)} \int\limits_\T \epsilon~dm  \\
 &\leq \dfrac{\epsilon}{\psi(1)}~\forall~n \geq K.
\end{align*}
By choosing $\epsilon = \psi(1)$, we have 
$$ \dfrac{1}{\psi(1)} \int\limits_\T \psi \left ( \dfrac{ \chi_{E_n}}{\lambda} \right ) dm \leq 1 ~\forall~n \geq K~\text{for all}~\lambda > 0.$$
Let $\lambda_k = \frac{1}{2^k}$ be a sequence such that $\lambda_k \in \left \{\lambda > 0 : \frac{1}{\psi(1)}\int_{\T} \psi \left( \frac{\chi_{E_n}}{\lambda} \right) dm \leq 1 \right \}$ for each $k \in \N$ and for all $n \geq K$. Hence, it follows that $\|\chi_{E_n}\|_\psi$ = 0 for all $n \geq K$.

The space $\mathcal{L}^\al = \{ f : \T \rightarrow \C $ measurable such that $\al(f) < \infty $\} is a Banach space under the continuous rotationally symmetric norm  $\al$. The $\al$-closure of $L^\infty$ is the general Lebesgue space $L^\al$. Some important facts about $L^\al$ and $\mathcal{L}^\al$ spaces established in \cite{Chen2014LebesgueAH} are that $L^\infty \subset L^\al \subset \mathcal{L}^\al \subset L^1$, and $||.||_1 \leq \al(.) \leq ||.||_\infty$. The space $L^\infty$ multiplies $L^\al$ back into $L^\al$ and the following inequality is satisfied: $\al(fg) \leq \|f\|_\infty \al(g)$ for all $f \in L^\infty$ and $g \in L^\al$. In fact this inequality holds for all $g\in \mathcal{L}^{\al}$.
\par
\smallskip
The $\al$-closure of $H^\infty(\T)$ is denoted by $H^\al$, and it is a closed subspace of $L^\al$. The general Hardy space $H^\al$ is a Banach space under the $\al$ norm. A simpler description of $H^\al$ has been proved in \cite{Chen2014LebesgueAH}, which is $H^\al = H^1 \cap L^\al$. Furthermore, note that $H^\infty \subset H^\al \subset H^1$.
\par
\smallskip
As we know $L^\al \subset L^1$, we can identify $f \in L^\al$ with a Fourier series
$$ f(z) = \sum_{j=- \infty}^{\infty} \hat{f}(j) z^j$$
where the Fourier coefficients are given by  $\hat{f}(j) = \int\limits_\T f(z) z^{-j} dm $, for all $j \in \Z$.\\
One of the characterizations of $H^\al$ obtained in \cite{Chen2014LebesgueAH} is:
$$ H^\al = \{ f \in L^\al \; : \; \hat{f}(j) = \int_{\T} f z^{-j} dm= 0,\;  \forall \; j < 0 \}.$$ 
Consider the $n$th Cesaro means of $f \in L^\al$, for each $n \geq 1$
\[ \sigma_n(f) = \frac{S_0(f) + S_1(f) +\cdots+ S_n(f)}{n+1} \] 
where $ S_k(f)$ stands for the $k$-th partial sum $\sum\limits_{j=-k}^{k} \hat{f}(j) z^j$, for $k \geq 0$.\\
The fact that $\sigma_n(f)$ converges to $f$ in $L^\al$ has been proved in \cite{Chen2014LebesgueAH}. For the convenience of the reader, we reproduce the result below. 

\begin{lem}\label{lemma2.2}
Let $\al$ be a continuous rotationally symmetric norm and $f\in L^\al$. Then $\al(\sigma_n(f) - f ) \to 0$ as $n \to \infty$, and also  $L^\al$ is the $\al$-closure of span$\{z^n : n \in \Z \}$.
\end{lem}
In order to obtain a characterization of simply invariant subspaces of $L^\al$ invariant under the multiplication by $z^n$, we require a decomposition of $L^\al$ in terms of the space $L^\al(z^n)$, defined as the $\al$ closure of $L^\infty(z^n)$. A similar decomposition has been proved in the context of $H^\al$ in \cite{singh2022multiplication}. Proceeding in a similar fashion as in the proof of Lemma 4.2 in \cite{singh2022multiplication}, we obtain a decomposition of $L^\al$ as follows.
\begin{lem}\label{lemma2.3}
Suppose $\al$ is a continuous rotationally symmetric norm on $L^\al$. Then
\begin{equation}\label{eq2.1}
L^\al =  L^\al(z^n) \oplus z L^\al(z^n) \oplus\cdots\oplus z^{n-1}  L^\al(z^n),
\end{equation}
where $\oplus$ is an algebraic direct sum.
\end{lem}
For a fixed $n \in \N$, the operator $T_{z^n}$ of multiplication by a monomial $z^n$ on $L^\al$ acts as an isometry. It can be verified as $\al(T_{z^n}(f)) = \al(z^n f) = \al(|z^n f|) = \al(|f|) = \al(f)$. A closed subspace $\mathcal{M}$ of $L^\al$ is simply invariant under $T_{z^n}$ if $z^n \mathcal{M} \subsetneq \mathcal M$, and is doubly invariant if $z^n \mathcal M = \mathcal M$. 
\par
Let us now introduce the definition of an $n$-unimodular matrix associated with an $r$-tuple of $L^\infty$ functions. It is a direct analogue of the definition of the $n$-inner matrix associated with an $r$-tuple of $H^\infty$ functions introduced in \cite{singh1997multiplication}.
\par
Let $(\varphi_1, \varphi_2,\ldots, \varphi_r)$ to be an $r$-tuple of $L^\infty$ functions $(r \leq n)$. For each $1 \leq j \leq r$, we can write $\varphi_j = \sum\limits_{i=1}^{n} z^{i-1} \; \varphi_{ji}$ where $\varphi_{ji} \in L^2(z^n)$. 
\par
Next, we define the $ r \times n$ matrix $A = (\varphi_{ji})$ for $1 \leq j \leq r$, $1 \leq i \leq n$, and call it the $n$-unimodular matrix associated with the tuple $(\varphi_1, \varphi_2,\ldots, \varphi_r)$, if
 $A A^*  =I$ almost everywhere. 
 
In particular, a function $\varphi \in L^\infty$  is $n$-unimodular if $\sum\limits_{i=1}^{n} \lvert \varphi_{i} \lvert ^2 = 1 \; a.e.$, where $ \varphi = \sum\limits_{i=1}^{n} z^{i-1} \; \varphi_{i}$ and $\varphi_i \in L^2(z^n)$. Observe that $\lvert \varphi_i \lvert^2 \leq 1 \; a.e.$ and hence $\varphi_i \in L^\infty(z^n)$. 

We will require a characterization for the associated matrix $A$ to be $n$-unimodular. This characterization is presented in Lemma \ref{lemma2.4} and is similar to the characterization of the $B$-inner matrix presented in \cite{singh1997multiplication}. We point out that the proof of the necessary part is very similar to that presented in \cite{singh1997multiplication}. However, the converse carries a different set of arguments. 

\begin{lem}\label{lemma2.4}
The matrix associated with an $r$-tuple $(\varphi_1, \varphi_2,\ldots, \varphi_r)$ of $L^\infty$ functions is $n$-unimodular if and only if $\{z^{kn} \varphi_j : k \in \Z,\;1 \leq j \leq r\}$ is an orthonormal set in $L^2$.
\end{lem}
\begin{proof}
Write $\varphi_j = \sum\limits_{i=1}^{n} z^{i-1} \; \varphi_{ji}$ for each $1 \leq j \leq r$ where $\varphi_{ji} \in L^2(z^n)$. Assume that $A = (\varphi_{ji})$ is an $n$-unimodular matrix. So, we have
\begin{equation}\label{eq2.3}
\sum\limits_{i=1}^{n} \lvert \varphi_{ji} \lvert ^2 = 1 \; a.e., \; \text{for every} \; \; 1 \leq j \leq r
\end{equation}
and
\begin{equation}\label{eq2.4}
\sum\limits_{i=1}^{n} \varphi_{ji} \overline{\varphi_{ki}} = 0 \; a.e. , \; \; \text{for}\; j \neq k, \; \text{and} \; 1 \leq j,k \leq r.
\end{equation}

The fact $\{z^{kn} \varphi_j : k \in \Z,\;1 \leq j \leq r\}$ is orthonormal in $L^2$ can be verified by following the arguments similar to those in the proof of Lemma 3.9 in \cite{singh1997multiplication}.

Conversely, let $\{z^{kn} \varphi_j : k \in \Z,\;1 \leq j \leq r\}$ be an orthonormal set in $L^2$. We shall show that $A=(\varphi_{ji})$ satisfies conditions (\ref{eq2.3}) and (\ref{eq2.4}). \\
 Write $\varphi_j = \sum\limits_{i=1}^{n} z^{i-1} \varphi_{ji}$. A simple calculation shows that $ \int\limits_\T \sum\limits_{i=1}^{n} \lvert\varphi_{ji}\lvert^2 dm = 1$ and $ \int\limits_\T \sum\limits_{i=1}^{n} \lvert\varphi_{ji}\lvert^2 z^{-kn} dm =0$ for all $k \neq 0$. Therefore, $\sum\limits_{i=1}^{n} \lvert\varphi_{ji}\lvert^2 = 1 \; a.e.$ This validates (\ref{eq2.3}). 

Note that when $j \neq l$ and $k \in \Z$, we have
\begin{equation*}
    \begin{split}
        0 &= \langle \varphi_j, z^{kn} \varphi_l \rangle  \\
        &= \sum\limits_{i=1}^{n} \langle \varphi_{ji}, z^{kn} \varphi_{li} \rangle \\
        &=   \left \langle  \sum\limits_{i=1}^{n} \varphi_{ji} \overline{\varphi_{li}}, z^{kn} \right \rangle 
    \end{split}
\end{equation*} 
Hence $\sum\limits_{i=1}^{n} \varphi_{ji} \overline{\varphi_{li}} = 0 \; a.e.$. This completes the proof of the lemma. 
\end{proof}

\section{Simply Invariant Subspaces of $L^\al$}\label{sec3}

In order to describe the simply invariant subspace $\mathcal M$ of $L^\al$ under the operator $T_{z^n}$, we shall first observe that $\mathcal M$ has a non trivial intersection with $L^\infty$ and that $\mathcal M\cap L^\infty$ is a weak*-closed subspace of $L^\infty$. We borrow the description of $T_{z^n}$--simply invariant subspaces of $L^\infty$ from \cite{SnehLata2010finite} and record it below for the convenience of the reader. The rest of the arguments in the proof of the main theorem are elementary and rely on the fact (Lemma \ref{lemma3.6}) that for any $f\in \mathcal M$ there exists an outer function $O$ such that $Of\in \mathcal M\cap L^\infty$.

\begin{thm}[\cite{SnehLata2010finite}]\label{thm3.1}
Let $\mathcal{M}$ be a weak*-closed subspace of $L^\infty$ which is simply invariant under $T_{z^n}$. Then the most general form of $\mathcal{M}$ is 
\[ \mathcal{M} = \sum\limits_{j=1}^{r} \oplus \varphi_j H^\infty(z^n) \oplus \mathcal{K}^{\overline{A}}_{\mathcal{M}}\;, \]
where \\
(a) Each $\varphi_j$ is an $n$-unimodular function such that $r \leq n$, \\
(b) $\overline{A} = (\overline{\varphi_{ji}}) \in M_{rn}(L^2(z^n))$, $\varphi_j = \sum\limits_{i=1}^{n} z^{i-1} \varphi_{ji}$,\\
(c) \;$\mathcal{K}^{\overline{A}}_{\mathcal{M}}$ = $\left \{ f \in \mathcal{M}: \sum\limits_{i=1}^{n} \overline{\varphi_{ji}}f_i = 0 \; a.e. \; \forall \;  1 \leq j \leq r \right \}$, where $f = \sum\limits_{i=1}^{n} z^{i-1} f_i$ and $f_i \in L^2(z^n)$.\\
Moreover, $\mathcal{K}^{\overline{A}}_{\mathcal{M}}$ is a doubly invariant subspace of $L^\infty$ under $T_{z^n}$. When $r = n$, $\mathcal{K}^{\overline{A}}_{\mathcal{M}} = \{0\}$. If $r < n$, then there exist infinitely many non-zero doubly invariant subspaces of $\mathcal{K}^{\overline{A}}_{{L}^\infty}$ when appended with $\sum\limits_{j=1}^{r} \oplus \varphi_j H^\infty(z^n)$ form a simply invariant subspace of $L^\infty$.
\end{thm}

\begin{rem}\label{rem3.2}
The construction of the characterization of simply $T_{z^n}$-invariant subspaces of $L^p$ ($0 < p \leq \infty$) described in \cite{SnehLata2010finite}, shows that the $r \times n$ matrix $A = (\varphi_{ji})$ associated with the tuple $(\varphi_1,\ldots,\varphi_r)$ is $n$-unimodular. 
\end{rem}
We justify Remark \ref{rem3.2} as follows:\\
The proof of Theorem A \cite{SnehLata2010finite} reveals that $\varphi_1,\ldots \varphi_r$ appearing in the statement of Theorem \ref{thm3.1} are orthonormal in $L^2$. Further, $z^{nm} \varphi_j \perp z^{nl} \varphi_k$ in $L^2$ for all $1 \leq j, k \leq r$ and $m, l \in \Z$. This leads to the fact that $\{ z^{kn} \varphi_j : k \in \Z, 1 \leq j \leq r\}$ is an orthonormal set in $L^2$. Now by Lemma \ref{lemma2.4}, we have $AA^*=I$.  
\par
Let us now state the main result of this section which helps us to describe the simply invariant subspaces of $L^\al$ invariant under $T_{z^n}$. 

\begin{thm}\label{thm3.3}
Let $\al$ be a continuous rotationally symmetric norm and $\mathcal{M}$ be a non-trivial closed subspace of $L^\al$ which is simply invariant under $T_{z^n}$. Then there exist $n$-unimodular functions $\varphi_1$,\ldots,$\varphi_r$ ($r \leq n$) such that
\[ \mathcal{M} = \sum\limits_{j=1}^{r} \oplus \varphi_j H^\al(z^n) \oplus \mathcal{K}^{\overline{A}}_{\mathcal{M}} \]
where (a) $H^\al(z^n) := \overline{H^\infty(z^n)}^{\al}$, \\
(b) A = $(\varphi_{ji})$ is an $r \times n$ matrix and each $\varphi_j = \sum\limits_{i=1}^{n} z^{i-1} \varphi_{ji}$ such that $\varphi_{ji} \in L^2(z^n)$,\\
(c) \;$\mathcal{K}^{\overline{A}}_{\mathcal{M}}$ = $\left \{ f \in \mathcal{M}: \sum\limits_{i=1}^{n} \overline{\varphi_{ji}}f_i = 0 \; a.e., \; \forall \;  1 \leq j \leq r \right \}$, where $f = \sum\limits_{i=1}^{n} z^{i-1} f_i$ and $f_i \in L^2(z^n)$.
\end{thm}

The following two lemmas, Lemma \ref{lemma3.4} and Lemma \ref{lemma3.5}, have been proved in \cite{singh2022multiplication} in the context of $ H^\al$. However, using similar arguments we obtain the $L^\al$ versions presented below.

\begin{lem}\label{lemma3.4}
Suppose $\al$ is a continuous rotationally symmetric norm and $\mathcal{M}$ be a closed subspace of $L^\al$. Then $ \mathcal{M} \cap L^\infty $ is weak*-closed in $L^\infty$.
\end{lem}

\begin{lem}\label{lemma3.5}
Let $\al$ be a continuous rotationally symmetric norm. Suppose $\mathcal{M}$ is a closed subspace of $L^\al$. Then, $\mathcal{M}$ is simply invariant under $T_{z^n}$ if and only if $\mathcal{M}$ is simply invariant under the algebra $H^\infty(z^n)$. 
\end{lem}

Our starting point in the proof of Theorem \ref{thm3.3} will be to guarantee that $\mathcal{M}$ has a non-trivial intersection with $L^\infty$. This fact is established in the next lemma.

\begin{lem}\label{lemma3.6}
Suppose $\mathcal{M}$ is a non-trivial $\al$-closed subspace of $L^\al$ such that $z^n \mathcal{M} \subsetneq \mathcal{M}$. Then, $\mathcal{M} \cap L^\infty \neq \{0\}$.
\end{lem}

\begin{proof}
Suppose $0 \neq f \in \mathcal{M} \subset L^\al \subset L^1$. Then $\lvert f \rvert ^{\frac{1}{2}} \in L^2$. In view of the decomposition (\ref{eq2.1}), we can write $\lvert f \lvert ^{\frac{1}{2}} = g_1 + zg_2+\cdots+z^{n-1}g_n$ for some  $g_1, g_2, \ldots,g_n \in L^2(z^n)$. Define
\[k_{j} =  exp (- \lvert g_j \rvert \; - i \; (\lvert g_j \rvert)^{\sim}) \]
where $(\lvert g_j \rvert)^{\sim}$ stands for the harmonic conjugate of the $L^2$ function $\lvert g_j\rvert$ and $(\lvert g_j \rvert)^{\sim} \in L^2(z^n)$ (this is possible for all $L^p$ functions \cite{Koosis1998HpSpace}). 

Note that $k_j$ is an analytic function such that $\lvert k_j \rvert \leq 1$, i.e, $\lvert k_j \rvert \leq 1 \in H^\infty(z^n)$  for each $1 \leq j \leq n$. Put  $O$ = $k_1 k_2 \ldots k_n$.\\
Note that
\begin{equation*}
\hspace{-1cm} O \lvert f \rvert^\frac{1}{2} = O \left (g_1 + z g_2+\cdots+ z^{n-1} g_n \right ) 
\end{equation*}
\[\hspace{0.8 cm} = O g_1 + z O g_2 +\cdots+ z^{n-1} O g_n. \]
\smallskip
Therefore,\\
\smallskip
$\lvert O \rvert \lvert f \rvert^\frac{1}{2}  \leq \lvert k_1 \rvert \lvert g_1 \rvert +\cdots+ \lvert k_n\rvert \lvert g_n \rvert$\\
\smallskip
\hspace{1cm}$=exp (- \lvert g_1 \rvert)\lvert g_1 \rvert + \cdots + exp (- \lvert g_n \rvert) \lvert g_n \rvert \leq n$.\\
Thus $O \lvert f \rvert^\frac{1}{2} \in L^\infty$ which implies that $O^2 f \in L^\infty$.
By Lemma \ref{lemma3.5}, we have $O^2 f \in \mathcal{M}.$ Hence $O^2 f \in \mathcal{M} \cap L^\infty$. 
\end{proof}

We now return to the proof of Theorem \ref{thm3.3}.
\begin{proof}
In view of Lemma~\ref{lemma3.4} and Lemma~\ref{lemma3.6}, we conclude that $\mathcal{M} \cap L^\infty$ is a non-trivial weak*-closed subspace of $L^\infty$. It can be easily seen that $\mathcal{M} \cap L^\infty$ is simply invariant under $T_{z^n}$. So, by Theorem~\ref{thm3.1}, there exist $n$-unimodular functions $\varphi_1,\ldots,\varphi_r$ ($r \leq n$) such that 
\begin{equation}\label{eq3.1}
\mathcal{M} \cap L^\infty = \sum\limits_{j=1}^{r} \oplus \varphi_j H^\infty(z^n) \oplus \mathcal{K}^{\overline{A}}_{\mathcal{M} \cap L^\infty}, 
 \end{equation}
and $A$ is the corresponding $n$-unimodular matrix in $M_{rn}(L^\infty(z^n))$ associated with the $r$ tuple $(\varphi_1,\ldots,\varphi_r)$. Moreover, the $T_{z^n}$-doubly invariant subspace $\mathcal{K}^{\overline{A}}_{\mathcal{M} \cap L^\infty}$ has the form 
\begin{equation}\label{eq3.2}
   \mathcal{K}^{\overline{A}}_{\mathcal{M} \cap L^\infty} = \left \{ f \in \mathcal{M}\cap L^\infty: \sum\limits_{i=1}^{n} \overline{\varphi_{ji}}f_i = 0 \; a.e. \; \forall \;  1 \leq j \leq r \right \}
\end{equation}
in which $f = \sum\limits_{i=1}^{n} z^{i-1} f_i$ and $f_i \in L^2(z^n)$.

We claim that $\mathcal{M} = \sum\limits_{j=1}^{r} \oplus \varphi_j H^\al(z^n) \oplus \mathcal{K}^{\overline{A}}_{\mathcal{M}}$.\\
It is trivial to note that  $\mathcal{K}^{\overline{A}}_{\mathcal{M}} \subset \mathcal{M}$.\\ Note that $\varphi_j H^\infty(z^n) \subset \mathcal{M}$ for each $1 \leq j \leq r$. We now show that $\varphi_j H^\al(z^n) \subset \mathcal{M}$. For any $f \in H^\al(z^n)$, there exists a sequence $\{f_n\}_{n=1}^{\infty} \in H^\infty(z^n)$ such that $\al(f_n - f) \to 0$ and since $\varphi_j \in L^\infty$, we see that $\varphi_j f_n$ converges to $\varphi_j f$ in $L^\al$. Also the sequence $\{\varphi_j f_n\}_{n=1}^\infty \subset \mathcal{M}$, and hence $\varphi_j f \in \mathcal{M}$. This implies that $\sum\limits_{j=1}^{r} \oplus \varphi_j H^\al(z^n) \oplus \mathcal{K}^{\overline{A}}_{\mathcal{M}} \subset \mathcal{M}$.

In order to prove the reverse containment, note that (as in the proof of Lemma \ref{lemma3.6}) for any $f \in \mathcal{M}$, we can construct an outer function $O \in H^\infty(z^n)$ such that $O f \in \mathcal{M} \cap L^\infty$. By equation (\ref{eq3.1}), we can write 
\begin{equation}\label{eq3.3}
 Of = \varphi_1 h_1 + \cdots + \varphi_r h_r + K
\end{equation}
where  $h_1, \ldots, h_r \in H^\infty(z^n)$ and $K \in \mathcal{K}^{\overline{A}}_{\mathcal{M} \cap L^\infty} $.

In view of Lemma~\ref{lemma2.3}, we can write
\begin{equation*}
    f = f_1 +\cdots+z^{n-1} f_n
\end{equation*}
for some $f_1,\ldots,f_n \in L^\al(z^n)$. 
Therefore, 
\begin{equation}\label{eq3.4}
    Of = Of_1 +\cdots+z^{n-1} O f_n
\end{equation}
Also we can decompose $\varphi_j$ and $K$ as follows 
\begin{equation}\label{eq3.5}
\begin{aligned}
    \varphi_j = \varphi_{j1} +\cdots+z^{n-1}\varphi_{jn}, \; \; \varphi_{ji} \in L^\infty(z^n) \\
    K = K_1 +\cdots+z^{n-1} K_n, \; \; K_i \in L^2(z^n)
\end{aligned}
\end{equation}

From equations (\ref{eq3.3}), (\ref{eq3.4}) and (\ref{eq3.5}), it is easy to see that
\begin{equation}\label{eq3.5}
\begin{pmatrix}
    O f_1 \\
    \vdots\\
    Of_n
\end{pmatrix}
\; = \;   \; \begin{pmatrix}
    \varphi_{11} h_1 + \varphi_{21} h_2 + \cdots + \varphi_{r1} h_r + K_1\\
    \vdots\\
    \varphi_{1n} h_1 + \varphi_{2n} h_2 + \cdots + \varphi_{rn} h_r + K_n
\end{pmatrix}
.
\end{equation}
It follows that
\begin{equation}\label{eq3.7}
\begin{pmatrix}
    \overline{O f_1} & \cdots & \overline{O f_n}
\end{pmatrix}
\; = \; \begin{pmatrix}
    \overline{\sum\limits_{j=1}^{r} \varphi_{j1} h_j + K_1} & \cdots & \overline{\sum\limits_{j=1}^{r} \varphi_{jn} h_j + K_n}
\end{pmatrix}
.\; 
\end{equation}

Therefore,
\[ \lvert O f_1\rvert^2 +\cdots+ \lvert O f_n\rvert^2 = \left \lvert \sum\limits_{j=1}^{r} \varphi_{j1} h_j + K_1 \right \rvert^2 + \cdots + \left \lvert \sum\limits_{j=1}^{r} \varphi_{jn} h_j + K_n \right \rvert^2 . \]
Hence, for each $1 \leq i \leq n$,
\[ \left \lvert \sum\limits_{j=1}^{r} \varphi_{ji} h_j + K_i \right \rvert^2 \leq \lvert O f_1\rvert^2 +\cdots+ \lvert O f_n\rvert^2 \leq \lvert O \rvert^2 (\lvert f_1 \rvert + \cdots +\lvert f_n \rvert)^2. \]
This yields
\[ \left \lvert \sum\limits_{j=1}^{r} \varphi_{ji} \dfrac{h_j}{O} + \dfrac{K_i}{O} \right \rvert \leq \lvert f_1 \rvert + \cdots +\lvert f_n \rvert. \]
Since $f_1,\ldots,f_n \in L^\al$, we have
\[\al\left( \left \lvert \sum\limits_{j=1}^{r} \varphi_{ji} \dfrac{h_j}{O} + \dfrac{K_i}{O} \right \rvert\right ) \leq \al (\lvert f_1 \rvert + \cdots +\lvert f_n \rvert) < \infty.\]
Therefore, for each $1 \leq i \leq n$, $ \mathlarger{\sum\limits_{j=1}^{r}} \varphi_{ji} \dfrac{h_j}{O} + \dfrac{K_i}{O}  \in L^\al$.

We now claim that each $\dfrac{h_j}{O}$ and $\dfrac{K_i}{O}$ belong to $L^\al$.\\
For a fixed $i$, we have
\[  \varphi_{1i} \frac{h_1}{O} + \cdots+  \varphi_{ji}  \frac{h_j}{O} + \cdots +  \varphi_{ri}  \frac{h_r}{O} +  \frac{K_i}{O} \in L^\al. \]
Since $\overline{\varphi_{ji}} \in L^\infty$, so we get
\[  \varphi_{1i} \overline{\varphi_{ji}} \frac{h_1}{O} + \cdots+  \lvert \varphi_{ji} \rvert^2 \frac{h_j}{O} + \cdots + \varphi_{ri} \overline{\varphi_{ji}} \frac{h_r}{O} +  \overline{\varphi_{ji}} \frac{K_i}{O} \in L^\al. \]
Take the summation over $1 \leq i \leq n$,
\[\sum\limits_{i=1}^{n} \varphi_{1i} \overline{\varphi_{ji}} \frac{h_1}{O} + \cdots+ \sum\limits_{i=1}^{n} \lvert \varphi_{ji} \rvert^2 \frac{h_j}{O} + \cdots + \sum\limits_{i=1}^{n} \varphi_{ri} \overline{\varphi_{ji}} \frac{h_r}{O} + \sum\limits_{i=1}^{n} \overline{\varphi_{ji}} \frac{K_i}{O} \in L^\al. \]
Since $A = (\varphi_{ji})$ is $n$-unimodular matrix, so by conditions  (\ref{eq2.3}), (\ref{eq2.4}) and (\ref{eq3.2}), we get $\dfrac{h_j}{O} \in L^\al$ for $1 \leq j \leq r$. So, $\mathlarger{\sum\limits_{j=1}^{r}} \varphi_{ji} \dfrac{h_j}{O} \in L^\al$. \\
Since $\mathlarger{\sum\limits_{j=1}^{r}} \varphi_{ji} \dfrac{h_j}{O} + \dfrac{K_i}{O} \in L^\al$, implies $\dfrac{K_i}{O} = \mathlarger{\sum\limits_{j=1}^{r}} \varphi_{ji} \dfrac{h_j}{O} + \dfrac{K_i}{O} - \mathlarger{\sum\limits_{j=1}^{r}} \varphi_{ji} \dfrac{h_j}{O} \in L^\al$. Hence the claim follows.

Furthermore, because $O$ is an outer function implies $\dfrac{h_j}{O} \in H^1$. Hence, $\dfrac{h_j}{O} \in H^\al$. By Lemma~\ref{lemma2.2}, the Cesaro means  $\sigma_l\left(\dfrac{h_j}{O}\right)$ converge to $\dfrac{h_j}{O}$ in $H^\al$. Since $\sigma_l\left(\dfrac{h_j}{O}\right)$ is a polynomial in $z^n$, so $\dfrac{h_j}{O}\in H^\al(z^n)$ for all $1 \leq j \leq r$.\\

Lastly to show that $f \in \sum\limits_{j=1}^{r} \oplus \varphi_j H^\al(z^n) \oplus \mathcal{K}^{\overline{A}}_{\mathcal{M}}$, it is suffices to prove that $\dfrac{K}{O} \in \mathcal{K}^{\overline{A}}_{\mathcal{M}}$. Observe that $\dfrac{K}{O} = f - \left (\varphi_1 \dfrac{h_1}{O} + \varphi_2 \dfrac{h_2}{O}+\cdots+ \varphi_r \dfrac{h_r}{O} \right) \in \mathcal{M}$ (since $\sum\limits_{j=1}^{r} \oplus \varphi_j H^\al(z^n) \subset \mathcal{M}$).\\
Further, as $K \in \mathcal{K}^{\overline{A}}_{\mathcal{M} \cap L^\infty}$ and in view of (\ref{eq3.2}), we have for all $1 \leq j \leq r$
$$ \overline{\varphi_{j1}}K_1 + \overline{\varphi_{j2}}K_2+\cdots+\overline{\varphi_{jn}}K_n = 0 ~  a.e. $$
Therefore, $ \overline{\varphi_{j1}} \dfrac{K_1}{O} + \overline{\varphi_{j2}}\dfrac{K_2}{O}+\cdots+\overline{\varphi_{jn}}\dfrac{K_n}{O} = 0 ~  a.e.$ This implies $\dfrac{K}{O} \in \mathcal{K}^{\overline{A}}_{\mathcal{M}}$ and hence
$$\mathcal{M} = \sum\limits_{j=1}^{r} \oplus \varphi_j H^\al(z^n) \oplus \mathcal{K}^{\overline{A}}_{\mathcal{M}}.$$
\end{proof}
\begin{rem}
Since $\|.\|_p$  for $1 \leq p < \infty$ and $p\neq 2 $ are continuous rotationally symmetric norms, so Theorem A in \cite{SnehLata2010finite} comes as an special case of Theorem \ref{thm3.3}.
\end{rem}

\begin{cor}(\cite{chen2017general})
Let $\al$ be a continuous rotationally symmetric norm and $\mathcal{M}$ be a non-trivial closed subspace of $L^\al$ which is simply invariant under $T_{z}$. Then there exists an unimodular function $\varphi$ such that
\[ \mathcal{M} = \varphi H^\al. \]
\end{cor}
\begin{proof}
For the case of multiplication by $T_z$, we have $n=1$. By Theorem \ref{thm3.3}, we see that there exists an unimodular function $\phi$ such that 
\[ \mathcal{M} = \varphi H^\al \oplus \mathcal{K}^{\overline{A}}_{\mathcal{M}}. \]
Here $A$ is the $1 \times 1$ matrix $(\varphi)$, and $\mathcal{K}^{\overline{A}}_{\mathcal{M}} = \{f \in \mathcal{M} : \overline{\varphi} f = 0~a.e.\}.$ Further, the unimodularity of $\varphi$ forces $\mathcal{K}^{\overline{A}}_{\mathcal{M}} = \{0\}$.
\end{proof}

\section{de Branges result in $H^\al$}\label{sec4}

de Branges in \cite{de2015square} first characterized Hilbert spaces which are contractively contained in $H^2$ and on which $T_z$ acts as an isometry. In \cite{UNSingh1991debranges}, it was established using the Wold decomposition that the condition of contractive containment can be dropped. This was a significant generalization of de Branges' result. This result was further extended to all $H^p$ spaces in  \cite{Agrawal1995deBranges}. In this section, we extend the result of Singh and Agrawal \cite{Agrawal1995deBranges} to Hardy spaces equipped with rotationally symmetric norms.

\begin{thm}\label{thm4.1}
Let $\mathcal{M}$ be a Hilbert space which is a vector subspace of $H^\al$ such that $T_z(\mathcal{M}) \subset \mathcal{M}$ and $T_z$ acts as an isometry on $\mathcal{M}$. Then, there exists $\phi \in H^2 \cap H^\al$, such that
\[ \mathcal{M} = \phi H^2.\]
Further, $\|\phi g\|_\mathcal{M} = \|g\|_{H^2}$ for all $g \in H^2$.
\end{thm}

\begin{proof}
Since $T_z$ acts as an isometry on $\mathcal{M}$, so by Wold decomposition theorem (\cite{Hoffman1962banach}), we can write
\begin{equation}\label{eq4.1}
 \mathcal{M} = \bigcap\limits_{n=0}^{\infty} T_{z^n} \mathcal{M} ~ \oplus ~ \bigoplus\limits_{n=0}^{\infty} T_{z^n} \mathcal{N}
\end{equation}
where $\mathcal{N}$ stands for the orthogonal complement of $z \mathcal{M}$ in $\mathcal{M}$.

We claim that $\bigcap\limits_{n=0}^{\infty} T_{z^n} \mathcal{M} = \{0\}$.\\ Note that for any $f \in \bigcap\limits_{n=0}^{\infty} T_{z^n} \mathcal{M}$ can be expanded as a Fourier series 
\[ f(z) = a_0 + a_1 z + a_2 z^2 + \cdots + a_n z^n + \cdots \]
where $a_n = \int\limits_\T f z^{-n} dm$. Since $f \in \bigcap\limits_{n=0}^{\infty} T_{z^n} \mathcal{M}$, so for $n \geq 0$ we can find $g_{n+1} \in \mathcal{M}$ such that $f(z)=z^{n+1} g_{n+1}(z)$. This forces $a_n = 0$ and hence $f=0$.\\
Therefore, 
\begin{equation}\label{eq4.2}
    \mathcal{M} = \mathcal{N} \oplus z \mathcal{N} \oplus z^2 \mathcal{N} \oplus \cdots.
\end{equation}

Let $\phi$ be any non-zero element of $\mathcal{N}$. Without loss of generality assume that  $\|\phi\|_\mathcal{M} = 1$.
We first show that $\phi$ multiplies $H^2$ into $\mathcal{M}$.
Suppose $g(z) = \sum\limits_{n=0}^{\infty} b_n z^n$ be an arbitrary element of $H^2$. Put $g_n(z) = \sum\limits_{k=0}^{n} b_k z^k$. \\
In view of (\ref{eq4.2}) we make the following computation, for any $n \geq 0$:
\begin{equation*}
\begin{split}
\|\phi g_n\|_\mathcal{M}^2 &= \|b_0 \phi +b_1 z \phi +\cdots+b_n z^n \phi \|^2_\mathcal{M}\\
&= \|b_0 \phi\|^2_\mathcal{M}+ \|b_1 z \phi\|^2_\mathcal{M}+\cdots+ \|b_n z^n \phi\|^2_\mathcal{M} \\
&= \lvert b_0\rvert^2 + \lvert b_1\rvert^2 + \cdots + \lvert b_n\rvert^2\\
&= \|g_n\|^2_{H^2}.
\end{split}
\end{equation*}
Since $\{g_n\}_{n=1}^\infty$ is a Cauchy sequence in $H^2$, so $\{\phi g_n\}_{n=1}^\infty$ is a Cauchy sequence in $\mathcal{M}$ and hence, there exists $h \in \mathcal{M}$ such that $\phi g_n \to h$ in $\mathcal{M}$.

Note that for any $k \leq n$, we have 
\begin{equation}\label{eq4.3}
    \begin{split}
        \phi g_n &= b_0 \phi +b_1 z \phi +\cdots + b_{k} z^{k} \phi + b_{k+1} z^{k+1} \phi + \cdots + b_n z^n \phi\\
        &=  b_0 \phi +b_1 z \phi +\cdots + b_{k} z^{k} \phi + z^{k+1} \phi h_n
    \end{split}
\end{equation}
where $h_n = b_{k+1} + b_{k+2} z +\cdots+ b_{n} z^{n-k-1} \in H^2$.\\
In a similar fashion as above, we can show that $\{\phi h_n\}$ is a Cauchy sequence in $\mathcal{M}$ and hence, there exists $f \in \mathcal{M}$ such that 
\begin{equation}\label{eq4.4}
    \phi h_n \to f  ~\text{in}~ {\mathcal{M}}. 
\end{equation} 
From (\ref{eq4.3}) and (\ref{eq4.4}), we get
$$h= b_0 \phi +b_1 z \phi +\cdots + b_{k} z^{k} \phi + z^{k+1} f.$$

Observe that the coefficient of $z^k$ in $h$ is equal to the coefficient of $z^k$ in $b_0 \phi +b_1 z \phi +\cdots + b_{k} z^{k} \phi$. Put $\phi = \beta_0 + \beta_1 z + \beta_2 z^2 + \cdots+ \beta_k z^k + \cdots$. Therefore, the coefficient of $z^k$ in $b_0 \phi +b_1 z \phi +\cdots + b_{k} z^{k} \phi+ \cdots$ is $b_0 \beta_k + b_1 \beta_{k-1} + \cdots + b_k \beta_0$. This is the same as the coefficient of $z^k$ in the formal product of $\phi g$. Therefore $\phi g= h \in \mathcal{M}$, and hence $\phi H^2 \subset \mathcal{M}$.

Recall that on the unit circle, $L^2 = H^2 \oplus \overline{z H^2}$. Let $f \in L^2$ then $f = f_1 + \overline{zf_2}$ for some $f_1, f_2 \in H^2$. So, $$\phi f = \phi f_1 + \phi\cdot\overline{z f_2}.$$ 
Note that $\phi f_1$ , $\phi f_2$ belong to $H^\al$, and this implies
$$\al(\phi\cdot\overline{z f_2}) = \al(\lvert \phi \cdot \overline{z f_2} \rvert) = \al(\lvert \phi f_2 \rvert) < \infty.$$
Thus $\al(\phi f) \leq \al(\phi f_1) + \al(\phi\cdot\overline{z f_2}) < \infty$, and hence $\phi$ multiplies $L^2$ into $\mathcal{L}^\al.$ \\
Note that $\phi f \in \mathcal{L^\al} \subset L^1$ for all $f \in L^2$. Therefore, by converse of Holder's inequality \cite{Leach1956holder's}, we have $\phi \in L^2$. Also $\phi \in \mathcal{M} \subset H^\al \subset H^1$, implies that $\phi \in L^2 \cap H^1 = H^2$. Finally, with this we conclude that $\phi \in H^2 \cap H^\al$ and $\mathcal{N} \subset H^2 \cap H^\al$. 

We claim that the \textit{dim}($\mathcal{N}$)= 1. \\
Let $\phi_1$ and $\phi_2$ be two unit vectors in $\mathcal{N}$ such that $\phi_1 \perp \phi_2$ in $\mathcal{M}$. We show that $ \phi_1 H^2 \perp \phi_2 H^2$ in $\mathcal{M}$. Consider $f , g \in H^2$ and $f_n = \sum\limits_{k=1}^{n} a_n z^n$, $g_n = \sum\limits_{k=0}^{n} b_n z^n$ such that $f_n \to f$ and $g_n \to g$ in $H^2$. We have already proved above that any element of $\mathcal{N}$ multiplies $H^2$ into $\mathcal{M}$ i.e., $\phi_1 f , \phi_2 g \in \mathcal{M}$ such that 
\begin{equation}\label{eq4.5}
    \phi_1 f_n \to \phi_1 f~\text{and}~\phi_2 g_n \to \phi_2 g~\text{in}~\mathcal{M}.
\end{equation}

Moreover, from the orthogonality of $\phi_1$ and $\phi_2$ in $\mathcal{M}$ and (\ref{eq4.2})
\begin{equation*}
    \begin{split}
        \langle \phi_1 f_n, \phi_2 g_n \rangle_\mathcal{M} &= \langle a_0 \phi_1 + a_1 z \phi_1 +\cdots+ a_n z^n \phi_1 , b_0 \phi_2 + b_1 z \phi_2 + \cdots+ b_n z^n \phi_2  \rangle_\mathcal{M}\\
         &= 0. 
    \end{split}
\end{equation*}
Taking limit $n \to \infty$, in view of (\ref{eq4.5}), we have
$$ \lim\limits_{n\to \infty} \langle \phi_1 f_n , \phi_2 g_n \rangle_{\mathcal{M}} = \langle \phi_1 f , \phi_2 g \rangle_{\mathcal{M}} = 0.$$
Therefore, $\phi_1 H^2 \perp \phi_2 H^2$ in $\mathcal{M}$.

Since $\mathcal{N} \subset H^2 \cap H^\al$, we see that $\phi_1 \phi_2 \in\phi_1 H^2 \cap \phi_2 H^2$. This forces $\phi_1 \phi_2 = 0$. Therefore, either $\phi_1 \equiv 0$ or $\phi_2 \equiv 0$. This confirms that \textit{dim}($\mathcal{N})=1$. 

Let $\mathcal{N} = \langle \phi \rangle$ for some unit vector $\phi \in \mathcal{N}$.
Therefore in view of (\ref{eq4.2}), we can write
\[\mathcal{M} = \phi H^2.\]
Further, since $\|\phi g_n\|_\mathcal{M}^2 = \|g_n\|_{H^2}^2$ and $\phi g_n \to \phi g$ in $\mathcal{M}$, it follows that $\|\phi g\|_\mathcal{M}^2 = \|g\|_{H^2}^2$ for all $g \in H^2$. This completes the proof.
\end{proof}

In the next theorem, we characterize the Hilbert space $\mathcal{M}$ under the condition that $H^\al$ is properly contained in $H^2$.  

\begin{thm}\label{thm4.2}
Let $\al$ be a continuous rotationally symmetric norm and $H^\al$ is properly contained in $H^2$. Suppose $\mathcal{M}$ is a Hilbert space algebraically contained in $H^\al$ such that $T_z$ acts as an isometry on $\mathcal{M}$ and $T_z \mathcal{M} \subset \mathcal{M}$. Then, $\mathcal{M} = \{0\}.$
\end{thm}

\begin{proof}
Since $\mathcal{M}$ satisfies the assumptions of Theorem \ref{thm4.1}, there exists $\phi \in H^2 \cap H^\al$ such that 
\[ \mathcal{M} = \phi H^2.\]

We claim that $\phi = 0$.\\
Let, if possible, $\phi \neq 0$.
The computations in the proof of Theorem \ref{thm4.1} show that $\phi$ multiplies $H^2$ into $H^\al$. Also, by the additional condition that $H^\al \subsetneq H^2$, we have $\phi H^2 \subset H^2$. This implies that $\phi \in H^\infty$.\\
For a fixed $n \geq 1$, define
\[ E_n = \left\{ e^{i\theta} : \lvert \phi(e^{i \theta}) \rvert > \frac{1}{n}\right\}.\]
Then
\[E_n^c = \left\{ e^{i\theta} : \lvert \phi(e^{i \theta}) \rvert \leq \frac{1}{n}\right\}.\]
Since $0 \neq \phi \in H^2$, $\phi$ cannot vanish on a set of Lebesgue measure zero, it follows that $m(\cap_{n=1}^\infty E_n^c)=0$. So, $m(E_n^c) \to 0$ as $n \to \infty$. Clearly, $\chi_{E_n^c} \to 0~a.e.$ and hence $\chi_{E_n} \to 1~a.e.$\\
Now for any $\epsilon > 0$, we can find $N_0 \in \N$ such that
\begin{equation}\label{eq4.6}
    1-\epsilon \leq \chi_{E_{N_0}} \leq 1+\epsilon~a.e.
\end{equation}

Also, $H^\al \subsetneq H^2$, there exists $h\in H^2$ and $h \notin H^\al$.
Define $g = \chi_{E_{N_0}} h $, then $g \in L^2$. We claim that $g \notin L^\al$.  \\
For if, $g \in L^\al$ then from (\ref{eq4.6}) we have, $\lvert (1-\epsilon) h \rvert \leq \lvert \chi_{E_{N_0}}h\rvert = \lvert g \rvert  ~a.e.$, this forces $h \in L^\al$. This is a contradiction.\\
Further, $\phi g = \chi_{E_{N_0}} \phi h$ belongs to $L^\al$ because $\phi h$ belongs to $H^\al$. Since $(1-\epsilon) \lvert \phi \rvert ~\leq~ \lvert \chi_{E_{N_0}} \phi \rvert ~a.e.$, we conclude that $\lvert \chi_{E_{N_0}} \phi \rvert$ is invertible on $\T$ except possibly on a set of measure zero.\\
Therefore, 
\begin{equation}\label{eq4.7}
\begin{split}
    \left \lvert \dfrac{1}{\chi_{E_{N_0}} \phi} \right \rvert &\leq \left \lvert \dfrac{1}{(1-\epsilon) \phi} \right \rvert ~~a.e.\\
    &\leq \dfrac{N_0}{1-\epsilon}~~a.e.
\end{split}
\end{equation}
In view of the invertibility of $\lvert \chi_{E_{N_0}} \phi \rvert$, we can 
write $ g  =  \dfrac{\chi_{E_{N_0}} \phi g}{\chi_{E_{N_0}} \phi}$, and hence
$$\lvert g \rvert \leq \dfrac{N_0}{1-\epsilon} \lvert \chi_{E_{N_0}} \phi g \rvert~a.e.$$
Therefore, $\al(g) \leq \dfrac{N_0}{1-\epsilon} \|\chi_{E_{N_0}}\|_\infty \al(\phi g) < \infty$ and $g \in L^\al$, which contradicts the fact that $g \notin L^\al$. This contradiction stems from the assumption that $\phi \neq 0$.
\end{proof}

\begin{rem}
An example of a continuous rotationally symmetric norm $\alpha$ for which $H^\al$ is contained in $H^2$ is provided in \cite{Pascal2009Compact}. In fact, \cite{Pascal2009Compact} constructs an example of $H^\al$ which is contained in all $H^p$ spaces for $1\le p<\infty$.
\end{rem}

We end this section with a corollary which is an immediate consequence of Theorems \ref{thm4.1} and \ref{thm4.2}. 

\begin{cor}(\cite{Agrawal1995deBranges})\label{cor4.4}
Let $\mathcal{M}$ be a Hilbert space which is algebraically contained in $H^p$ for any $1 \leq p \leq \infty$. Further assume that the operator $T_z$ acts as an isometry on $\mathcal{M}$ and $T_z \mathcal{M} \subset \mathcal{M}$. Then
\[ \mathcal{M} = b H^2\]
for a unique $b$:\\
(1) If $1 \leq p \leq 2$, $b \in H^{2p/2-p}$ and  $\|b f\|_M = \|f\|_{H^2}$ for all $f\in H^2$.\\
(2) If $p > 2$, b=0.
\end{cor}

\begin{rem}
Note that the part (2) of Corollary \ref{cor4.4} is in line with the conclusion of Theorem \ref{thm4.2}.
\end{rem}

\nocite{Hoffman1962banach}
\bibliographystyle{spmpsci}
\bibliography{references.bib}

\end{document}